\documentclass[11pt]{amsart}
\newcommand{\K}{\mathcal K}
\newcommand{\C}{\mathcal C}
\newcommand{\D}{\mathcal D}
\newcommand{\cls}[3]{\overline{#1}^{\scriptscriptstyle {L^{#2}(#3)}}}

\usepackage{amssymb}
\usepackage{enumerate}
\usepackage{amsfonts}
\usepackage{amsmath}
\usepackage{mathrsfs}
\allowdisplaybreaks[1]

\newcounter{fig}

\newcommand{\mybic}{\author{Gianluca Cassese}
                     \address{Universit\`{a} Milano Bicocca}
                     \email{gianluca.cassese@unimib.it}
                     \curraddr{Department of Statistics, Building U7, Room 2097, via Bicocca 
                               degli Arcimboldi 8, 20126 Milano - Italy}}

\newtheorem{theorem}{Theorem}
\theoremstyle{plain}
\newtheorem{acknowledgement}{Acknowledgement}
\newtheorem{corollary}{Corollary}

\newtheorem{example}{Example}

\newtheorem{lemma}{Lemma}

\newtheorem{remark}{Remark}

\newcommand{\Prob}{\mathbb{P}}

\newcommand{\Sim}{\mathscr{S}} 
\DeclareMathOperator*{\LIM}{LIM} 
\DeclareMathOperator*{\co}{co}

\newcommand{\B}{\mathfrak{B}} 
\newcommand{\A}{\mathscr{A}} 

\newcommand{\R}{\mathbb{R}} 
\newcommand{\N}{\mathbb{N}}

\newcommand{\abs}[1]{\vert #1\vert} 
\newcommand{\dabs}[1]{\left\vert #1\right\vert}

\newcommand{\net}[3]{\left\langle #1_{#2}\right\rangle_{#2\in #3} } 
\newcommand{\neta}[1]{\net{#1}{\alpha}{\mathfrak A}} 
\newcommand{\nnet}[3]{\left\langle #1\right\rangle_{#2\in #3} } 
 
\newcommand{\seq}[2]{\net{#1}{#2}{\mathbb{N}}} 
\newcommand{\sseq}[2]{\nnet{#1}{#2}{\mathbb{N}}} 
\newcommand{\seqn}[1]{\seq{#1}{n}} 
\newcommand{\sseqn}[1]{\sseq{#1}{n}} 
\newcommand{\seqnk}[1]{\sseq{#1_{n_k}}{k}} 

\newcommand{\norm}[1]{\Vert #1\Vert}

\newcommand{\set}[1]{\mathbf{1}_{#1}}

\newcommand{\emp}{\varnothing}
\newcommand{\0}{\emptyset}

\begin{document}

\title[convergence in measure]{Convergence in Measure under Finite Additivity}
\mybic
\date
\today


\begin{abstract}
We investigate the possibility of replacing the topology of convergence in 
probability with convergence in $L^1$, upon a change of the underlying 
measure under finite additivity. We establish conditions for the continuity of 
linear operators and convergence of measurable sequences, including a finitely 
additive analogue of Koml\'{o}s Lemma. We also prove several topological
implications. Eventually, a characterization of continuous linear 
functionals on the space of measurable functions is obtained.
\end{abstract}

\maketitle

\section{Introduction and Notation}

This paper investigates some properties of the space $L^0(\lambda)$ of  
$\lambda$-measurable, real valued functions on some set $\Omega$, where
$\lambda$ is a bounded, finitely additive set function defined on an algebra
$\A$ of subsets of $\Omega$, i.e. $\lambda\in ba(\A)$. We first characterize
in section \ref{sec riesz} the dual space of $L^0(\lambda)$ and study some of 
its properties, particularly positivity. In the following section \ref{sec bdd} we 
investigate several boundedness conditions and, in sections \ref{sec topology} and
\ref{sec operators} we develop some topological implications including conditions
for continuity of linear operators.  Eventually, in section \ref{sec convergence} we 
study convergence properties of sequences. Section \ref{sec riesz} is quite 
independent from the following ones.

Although being an entirely standard and widely used concept in probability and 
mathematical statistics, convergence in measure is much less popular in analysis, 
even assuming countable additivity. It is known that the corresponding topology is 
completely metrizable but, in general, not separable; moreover, it is not 
linear so that some useful tools such as separation theorems are not available. 
Actually, even a characterization of continuous linear functionals is missing. Finite 
additivity introduces additional complications inducing, e.g., incompleteness.

The main idea of this paper is to show that some of the techniques developed 
in the classical setting are still available under finite additivity, by a change of the
given measure. In particular we show that, upon replacing the original measure $\lambda$
with another suitably chosen but \textit{near} to it, $\mu$, the topology of convergence 
in $\lambda$-measure may be replaced by the $L^1(\mu)$ topology. Our analysis
focuses on bounded, convex sets of measurable functions. Convexity is a crucial
property for our technique but is delicate as the topology of convergence in measure
is not locally convex. We prove in Theorem \ref{th co} that $L^0(\lambda)$ is
a locally convex topological vector space if and only if $\lambda$ is strongly
discontinuous, a property defined in Lemma \ref{lemma cts}. The main result,
Theorem \ref{th bdd}, shows that bounded, convex subsets of $L^0(\lambda)$
which admit a lower bound are actually bounded in $L^1(\mu)$. We draw
from this conclusion a number of implications. In Theorem \ref{th sep} we
obtain a set for which $L^0$ and $L^1$ closures coincide while in 
Corollaries \ref{cor TVS} and \ref{cor lattice} conditions under which continuous, 
$L^0(\lambda)$ valued operators are continuous as maps on $L^1(\mu)$. Likewise, 
Theorem \ref{th memin} proves, that a $\lambda$-convergent sequence admits
a subsequence converging in $L^1(\mu)$. We also obtain in Theorem 
\ref{th komlos} a partial, finitely additive analogue of the celebrated lemma of Koml\'os.
We make use of some results developed in a related paper, \cite{lebesgue}.
Some of the results obtained here have a countably additive counterpart and, as always,
a possible approach would then be to pass through the Stone space representation (see
e.g. \cite{fefferman}). We find, however, that even when this possibility is available, 
working directly under finite additivity is preferable as it gives explicit constructions 
rhater than isomorphic ones.

In the notation as well as in the terminology on finitely additive measures and integrals
we mainly follow Dunford and Schwarz \cite{bible}. We prefer, though, the symbol
$\abs\mu$ of \cite{rao} to denote the total variation measure associated with 
$\mu\in ba(\A)$. The space $ba(\A)$ is endowed with the usual lattice structure
described, e.g., in \cite{rao} and we thus use the lattice symbols $\mu^+=\mu\vee0$ 
and $\mu^-=-(\mu\wedge0)$ and the fact that $\abs\mu=\mu^++\mu^-$. The 
integral of $f\in L^1(\mu)$ will be denoted at will as $\int fd\mu$ or $\mu(f)$ but 
always as $\mu_f$ when considered itself as a set function.

We consider some special subfamilies of $ba(\A)$, in particular the family
$ba_0(\A)$ of set functions on $\A$ with finite range and
$ba(\lambda)=\{\mu\in ba(\A):\mu\ll\lambda\}$. Moreover, we denote by 
(\textit{i}) $ba_0(\lambda)$, (\textit{ii}) $ba_\infty(\lambda)$ and (\textit{iii}) 
$ba_*(\lambda)$ the classes of those set functions $\mu\in ba(\lambda)$ such 
that (\textit{i}) $\mu$ has finite range, (\textit{ii}) $\abs\mu\le c\abs\lambda$ 
for some $c\in\R_+$ and (\textit{iii}) $\mu\in ba_\infty(\lambda)$ and $\abs\mu(A)=0$ 
if and only if $\abs\lambda(A)=0$, respectively. In the above defined families the 
symbol $ba$ will be replaced by $\Prob$ to indicate the intersection of the corresponding 
family with the set $\Prob(\A)$ of finitely additive probability measures on $\A$.

The linear space of $\A$-simple functions, generated by the indicators of sets in 
$\A$, will be indicated by $\Sim(\A)$ and, when considered as a normed space, 
will always be endowed with the supremum norm. A sequence $\seqn f$ in 
$L^0(\lambda)$ $\lambda$-converges to $f\in\R^\Omega$ if 
$\lim_n\abs\lambda^*(\abs{f_n-f}>c)=0$ for any $c>0$, in which case 
$f\in L^0(\lambda)$ too. In fact $f\in L^0(\lambda)$ if and only if there exists a 
sequence $\seqn f$ in $\Sim(\A)$ which $\lambda$-converges to $f$.
As in \cite[II.1.11]{bible}, a (not necessarily measurable) function $f:\Omega\to\R$ 
is said to be $\lambda$-null if $\abs\lambda^*(\abs f>c)=0$ for all $c>0$ and
a subset of $\Omega$ is $\lambda$-null when its indicator function is null. A
function $f:\Omega\to\R$ possesses some property $\lambda$-a.s. -- e.g.
$f\ge0$ $\lambda$-a.s. -- if there exists a $\lambda$-null function $g$ such
that $f+g$ possesses that property. Given its use in the sequel, we say that
$f$ is a $\lambda$-a.s. lower bound of a set $\K$ if for each $c>0$ and
$k\in\K$ we have $\abs\lambda^*(k<f-c)=0$.

The set $L^0(\lambda)$ of measurable functions is endowed with the metric
\begin{equation}
\label{metric}
d(f,g)=\inf\left\{c+\abs\lambda^*(\abs{f-g}\ge c):c>0\right\} 
\end{equation}
(or equivalently with $\rho(f,g)=\int(\abs{f-g}\wedge1)d\abs\lambda$). By a 
bounded subset $\K$ of $L^0(\lambda)$ we will always mean a subset which, upon
delation, is contained in any ball around the origin. This definition turns out to be
equivalent to the condition
\begin{equation}
\lim_{c\to\infty}\sup_{f\in\K}\abs\lambda^*(\abs f>c)=0
\end{equation}
Of course, if $\K_1$ and $\K_2$ are convex, bounded subsets of $L^0(\lambda)$ 
then from
\begin{align*}
\abs\lambda^*(\abs{af+(1-a)g}>c)
&\le
\abs\lambda^*(\abs f>c)+\abs\lambda^*(\abs g>c)
\end{align*}
we deduce that $\co(\K_1\cup\K_2)$ is itself bounded.
Any space $X$ of measurable functions mentioned in this paper, including 
$L^0(\lambda)$ and $\Sim(\A)$, will be endowed with pointwise 
ordering in terms of which $f\ge g$ is synonymous to $f(\omega)\ge g(\omega)$ 
for all $\omega\in\Omega$. The symbol $X_+$ should be interpreted
accordingly. The closure of a set $H$ in $L^p(\lambda)$ will be denoted as
$\cls Hp\lambda$.

We will use repeatedly the following, finitely additive version of Tchebycheff inequality
where $f\in L^1(\lambda)_+$:
\begin{align}
\label{tcheby}
\abs\lambda^*(f>c)&=\sup m(f>c)\le c^{-1}\sup m(f)\le c^{-1} \abs\lambda(f)
\end{align}
where the supremum is computed over all $m\in ba(\sigma\A)_+$ which are 
extensions of $\abs\lambda$, see \cite[3.3.3]{rao}.

Eventually, if $X$ and $Y$ are vector lattices, a linear map $T:X\to Y$ defines an 
order bounded operator if for all sets of the form, $U=\{x\in X:x_1\le x\le x_2\}$ 
there exists $y\in Y$ such that $\abs{Tx}\le\abs y$ for all $x\in U$. If $A$ is any
set, we denote by $\co(A)$ its convex hull.

\section{Linear Functionals on $L^0(\lambda)$}
\label{sec riesz}

$\lambda\in ba_0(\A)$ if and only if it may be written as a finite sum 
$\sum_{n=1}^N\alpha_n\lambda_n$ where, $\lambda_1,\ldots,\lambda_N\in ba(\A)$ 
take their values in the set $\{0,1\}$, \cite[Lemma 11.1.3]{rao}.  Other
useful properties are proved in the next

\begin{lemma}
\label{lemma ba0}
The following properties are equivalent: (i) $\lambda\in ba_0(\A)$, 
(ii) $\abs\lambda\in ba_0(\A)$, (iii) there exists $\eta>0$ such that
$A\in\A$ and $\abs\lambda(A)>0$ imply $\abs\lambda(A)>\eta$,
(iv) there exists $c>0$ such that $A\in\A$ and $\abs{\lambda(A)}>0$ 
imply $\abs{\lambda(A)}>c$.
\end{lemma}

\begin{proof}
By construction, the range of $\lambda^+$ (resp. $\lambda^-$) is contained 
in that of $\lambda$ (resp. $-\lambda$) so that the range of 
$\abs\lambda=\lambda^++\lambda^-$ is finite if that of $\lambda$ is so. The 
implication (\textit{ii})$\Rightarrow$(\textit{iii}) is obvious. Let $\eta$ be as in 
(\textit{iii}) and let $A,B\in\A$ be such that $\lambda^+(A)>0$ and that 
$\lambda^+(B^c)+\lambda^-(B)<(\lambda^+(A)\wedge\eta)/2$. Then,
$\abs\lambda(A\cap B)\ge\lambda^+(A)-\lambda^+(B^c)\ge\lambda^+(A)/2>0$
so that $\abs\lambda(A\cap B)>\eta$, by (\textit{iii}), and so
\begin{align*}
\lambda^+(A)\ge\lambda^+(A\cap B)
\ge
\abs\lambda(A\cap B)-\lambda^-(B)
>
\eta/2
\end{align*}
But then, by \cite[Proposition 11.1.5]{rao}, both $\abs\lambda$ and $\lambda^+$
have finite range and the same must be true of $\lambda$, which implies (\textit{iv}).
The implication (\textit{iv})$\Rightarrow$(\textit{i}) is again a consequence of
\cite[Proposition 11.1.5]{rao}.
\end{proof}

Thus $ba_0(\A)$ is a vector sublattice of $ba(\A)$. Moreover,

\begin{lemma}
$ba_0(\A)_+$ is a convex, extreme subset of $ba(\A)_+$ with the property that
$\mu\ll m$ and $m\in ba_0(\A)$ imply $\mu\in ba_0(\A)$.
\end{lemma}

\begin{proof}
The first property is obvious since $ba_0(\A)$ is a vector space. Choose
$\lambda_1,\lambda_2\in ba_0(\A)_+$ and $0<t<1$ such that 
$\mu=t\lambda_1+(1-t)\lambda_2\in ba_0(\A)_+$. There exists then $\eta$
such that $A\in\A$ and $\mu(A)<\eta$ imply $\mu(A)=0$. Suppose that
$B\in\A$ is such that $(\lambda_1\vee\lambda_2)(B)<\eta$. Then,
$\mu(B)=0$ and thus $\lambda_1(B)=\lambda_2(B)=0$ and the same is
true of any $C\in\A$ such that $C\subset B$. But then
$
(\lambda_1\vee\lambda_2)(B)
=
\sup_{\{C\in\A:C\subset B\}}
\lambda_1(C)+\lambda_2(B\backslash C)
=0
$.
We conclude that $\lambda_1\vee\lambda_2\in ba_0(\A)$ and thus $\mu\in ba_0(\A)$ 
because $\mu\le\lambda_1\vee\lambda_2$. Assume that $\mu\ll m\in ba_0(\A)$. By
Lemma \ref{lemma ba0} there is no loss of generality in assuming $m,\mu\ge0$.
Let $\mu_m^c+\mu_m^\perp$ be the Lebesgue decomposition of $\mu$, with 
$\mu_m^c\ll m$ and $\mu_m^\perp\perp m$. By the first part of this Lemma, 
$\mu_m^c,\mu_m^\perp\in ba_0(\A)$. Thus for some $\eta>0$, $m(A)<\eta$ 
implies $\mu_m^c(A)=0$ and thus $m(A)=0$ as $m\ll\mu_m^c$.
\end{proof}

\begin{lemma}
\label{lemma lattice}
Continuous linear functionals on $L^0(\lambda)$ form a vector lattice. 
\end{lemma}

\begin{proof}
Indeed, if 
$f\in L^0(\lambda)$ the set $\mathscr U(f)=\{g\in L^0(\lambda):\abs{g}\leq\abs f\}$ 
is bounded in $L^0(\lambda)$ and so is any order bounded set 
$[h,f]=\{g\in L^0(\lambda):h\le g\le f\}$ given the inclusion 
$[h,f]\subset h+\mathscr U(f-h)$. Any continuous linear functional $\phi$ on 
$L^0(\lambda)$ is thus order bounded and the claim follows from 
\cite[Theorem 1.13]{aliprantis burkinshaw}. 
\end{proof}

\begin{theorem}
\label{th riesz}
There exists a linear isomorphism between the space of continuous linear functionals 
on $L^0(\lambda)$ and the space $ba_0(\lambda)$ and this is defined implicitly 
via the identity
\begin{equation}
\label{riesz}
\phi(f)=\int fd\mu\qquad f\in L^0(\lambda)
\end{equation}
\end{theorem}

\begin{proof} 
By Lemma \ref{lemma lattice} there is no loss of generality in assuming, as 
we shall do henceforth, that $\phi$ is positive. By  \cite[Theorem 1]{JMAA}
we have the representation
\begin{equation}
\label{representation}
\phi(f)=\phi^\perp(f)+\int fd\mu
\end{equation}
where $\phi^\perp$ is a positive linear functional on $L^0(\lambda)$ with 
$\phi^\perp(1)=0$ and $\mu\in ba_+$ is such that 
$L^0(\lambda)\subset L^1(\mu)$. Thus $\phi^\perp(f\wedge n)=0$ for all 
$f\in L^0(\lambda)_+$ so that
\begin{align*}
\phi(f)
=
\lim_n\phi(f\wedge n)
=
\lim_n\int (f\wedge n)d\mu
=
\int fd\mu
\end{align*}
as a consequence of the fact that $f\wedge n$ converges to $f$ in $L^0(\lambda)$
and in $L^1(\mu)$ since $f\in L^1(\mu)$. $\mu\ll\lambda$ is a consequence of 
$\phi$ being continuous. Suppose that for each $n\in\N$ there is $H_n\in\A$ such 
that $0<\mu(H_n)\leq2^{-n}$ and let $G_k\in\A$ be such that 
$\mu(G_k^c)+\lambda_\mu^\perp(G_k)<2^{-k}$, 
with $\lambda_\mu^\perp$ being the part of $\lambda$ orthogonal to $\mu$ 
emerging from Lebesgue decomposition. Then, choosing the integer $k_n$ 
large enough and $H'_n=H_n\cap G_{k_n}$, we have 
$0<\mu(H_n)-2^{-k_n}\le\mu(H'_n)\leq2^{-n}$ and $\lim_n\lambda(H'_n)=0$. 
If $f_n=\mu(H'_n)^{-1}\set{H'_n}$ then $\seqn{f}$ $\lambda$-converges to $0$ 
but $\phi(f_n)=\int f_nd\mu=1$, a contradiction. We conclude that for $n$ large 
enough $\mu(A)\leq2^{-n}$ implies $\mu(A)=0$, i.e. $\mu\in ba_0(\lambda)$. 
Conversely, assume that $\mu\in ba_0(\lambda)$ and that $U\subset L^0(\lambda)$ 
is bounded in $L^0(\lambda)$. Then, choosing $\delta$ accurately, 
$\abs\mu^*(\abs f>\delta)=0$ for all $f\in U$ so that 
$\sup_{f\in U}\dabs{\int fd\mu}\le\delta\norm\mu$ and thus the 
right hand side of \eqref{riesz} defines a bounded linear functional on 
$L^0(\lambda)$.
\end{proof}

By Theorem \ref{th riesz}, $L^0(\lambda)\subset L^1(\mu)$ upon a change 
of the underlying measure. The inclusion $\mu\in ba_0(\lambda)$ implies also 
that a set bounded in $L^0(\lambda)$ is necessarily bounded in $L^1(\mu)$ or 
even in $L^\infty(\mu)$. Moreover, if $m$ is countably additive, then so is $\mu$. 
However, $\lambda$ and $\mu$ may be very far from one another, e.g. for what 
concerns null sets.

A linear functional $\phi$ on $L^0(\lambda)$ is strictly positive if it is positive and 
if $f\in L^0(\lambda)_+$ and $\abs\lambda^*(f>c)>0$ for some $c>0$ imply 
$\phi(f)>0$.

\begin{corollary}
\label{corollary positive}
$L^0(\lambda)$ admits a continuous, strictly positive linear functional if and only if 
$\lambda\in ba_0(\A)$.
\end{corollary}
\begin{proof}
In fact, if $\phi$ and $\mu$ are related via \eqref{riesz} then in order for $\phi$ to 
be strictly positive one should have $\mu\ge0$ and $\abs\lambda(A)=0$ if and only 
if $\mu(A)=0$. But this together with $\mu\in ba_0(\lambda)$ implies the 
existence of $\delta>0$ such that $\abs\lambda(A)<\delta$ implies $\mu(A)=0$ 
and thus $\abs\lambda(A)=0$. We conclude that $\lambda\in ba_0(\A)$, by Lemma 1. 
On the other hand, if $\lambda\in ba_0(\A)$ then the integral $\int fd\abs\lambda$ is 
well defined for all $f\in L^0(\lambda)$ and strictly positive as $f\in L^0(\lambda)_+$ 
and $\int fd\abs \lambda=0$ imply $\abs\lambda^*(f>c)=0$ for all $c>0$, i.e. $f$ is 
$\lambda$ null.
\end{proof} 

Thus if $\lambda\notin ba_0(\A)$ there does not exist any strictly positive linear 
functional, so that if $\mu$ is as in Theorem \ref{th riesz} one necessarily has sets 
$A\in\A$ such that $\lambda(A)>0=\mu(A)$. The goal of the next section will be to 
find $\mu\in ba(\A)$ that guarantees the integrability of some subset of $L^0(\lambda)$ 
without affecting null sets.

We provide as a last result a different proof of \cite[Theorem 1]{mukherjee summers}%
\footnote{I am in debt with an anonymous referee for calling my attention on the
paper of Mukherjee and Summers \cite{mukherjee summers}.}.

\begin{corollary}[Mukherjee and Summers]
Let $\A$ be a $\sigma$-algebra, $\lambda\in ca(\A)$ and let $\phi$ be a continuous
linear functional on $L^0(\lambda)$. Then either $\phi=0$ or $\lambda$ has atoms.
\end{corollary}

\begin{proof}
Let $\mu\in ba_0(\lambda)_+$, $\mu\ne0$. Under the current assumptions, $\mu$ 
admits a Radon Nikodym derivative $f_\mu\in L^1(\abs\lambda)$, moreover $\mu$ has 
atoms. Let $\eta>0$ be such that $A\in\A$ and $\mu(A)<\eta$ imply $\mu(A)=0$. 
Let also $c<\eta/\norm\lambda$. Then $\mu(f_\mu<c)\le c\abs\lambda(f_\mu<c)<\eta$ 
so that $\mu(f_\mu<c)=0$. If $A'\in\A$ is an atom of $\mu$ then so is 
$A=A'\cap\{f_\mu\ge c\}$. If $B\in\A$ and $B\subset A$ then 
$\abs\lambda(B)\le c^{-1}\mu(B)$ so that either $\abs\lambda(B)=0$ or 
$\abs\lambda(A\backslash B)=0$. Thus $A$ is an atom for $\lambda$ too and the claim
follows from Theorem \ref{th riesz}.
\end{proof}

Mukherjee and Summers prove this claim using the fact that if $\lambda\in ca(\A)$ 
is atomless then its range is convex, a fact which is simply not true under finite additivity, 
see the examples in \cite[p. 143]{rao}. To prove a corresponding version we need the 
decomposition of Sobczyk and Hammer, see \cite[Theorem 5.2.7 and Remark 5.2.8]{rao}: 
each $\lambda\in ba(\A)$ decomposes uniquely as
\begin{equation}
\label{sob}
\lambda=\lambda_0+\sum_na_n\lambda_n
\end{equation}
where $\lambda_0$ is strongly continuous (i.e. such that for each $\varepsilon>0$
there exists a finite partition ${A_1,\ldots,A_N}\subset\A$ such that 
$\sup_n\abs{\lambda_0}(A_n)<\varepsilon$), the $\lambda_n$'s are distinct and 
$\{0,1\}$-valued, $a_n\ne0$ and $\sum_n\abs{a_n}<\infty$.

\begin{theorem}
$ba_0(\lambda)\ne\{0\}$ if and only $\lambda$ is not strongly continuous.
\end{theorem}

\begin{proof}
If $\lambda$ is not strongly continuous, $ba_0(\lambda)$ contains each $\{0,1\}$-%
valued component of $\lambda$ in the decomposition \eqref{sob}. Conversely, if $\lambda$ 
is strongly continuous and $\mu\in ba_0(\lambda)$, then for $\varepsilon>0$ small 
enough so that $0<\abs\lambda(A)<\varepsilon$ implies $\abs\mu(A)=0$ we find an 
$\A$-measurable finite partition of $\Omega$ on which $\abs\mu$ vanishes so that 
$\mu=0$.
\end{proof}

We close by proving a result that, despite being pretty straightforward, will be
useful in the sequel.

\begin{lemma}
\label{lemma cts}
Each $\lambda\in ba(\A)$ decomposes uniquely as
\begin{equation}
\label{cts}
\lambda=\lambda_c+\lambda_d
\end{equation}
where $\lambda_c$ is strongly continuous and $\lambda_d$ is strongly discontinuous,
i.e. orthogonal to any strongly continuous element of $ba(\A)$. Moreover, $\lambda_c$
coincides with $\lambda_0$ in \eqref{sob}.
\end{lemma}

\begin{proof}
The decomposition follows from that of Bochner and Phillips. Strongly continuous
elements of $ba(\A)$ form in fact a vector sublattice of $ba(\A)$ by 
\cite[Proposition 5.1.8]{rao}. If $m,\mu\in ba(\A)$ are such that $m$ is strongly 
continuous, and thus $\abs m$, and $\abs\mu\le\abs m$ then it is obvious from 
the definition that $\mu$ is strongly continuous too. If $\neta m$ is an increasing net 
of strongly continuous elements of $ba(\A)$ and if $m=\lim_\alpha m_\alpha\in ba(\A)$, 
then fix $\varepsilon$ and find $\alpha\in\mathfrak A$ sufficiently large so that 
$\norm{m-m_\alpha}=(m-m_\alpha)(\Omega)<\varepsilon/2$ and, by the 
properties of $m_a$, a partition $\{A_1,\ldots,A_N\}\subset\A$ such that 
$\sup_n\abs{m_a}(A_n)<\varepsilon/2$. But then, 
$\sup_n\abs{m}(A_n)<\varepsilon/2+\sup_n\abs{m_a}(A_n)<\varepsilon$. This proves
existence and uniqueness of \eqref{cts}. Applying \eqref{sob} to $\lambda_d$ 
we conclude, by orthogonality, that $\lambda_d$ is necessarily of the form 
$\sum_nb_n\lambda_{d,n}$ with the $\lambda_{d,n}$'s 0-1 valued and distinct and 
with $\sum_n\abs{b_n}<\infty$. But then, the claim follows from uniqueness of the 
decomposition \eqref{sob}.
\end{proof}

\section{Bounded Subsets of $L^0(\lambda)$}
\label{sec bdd}

In this section we provide conditions under which bounded subsets of $L^0(\lambda)$ 
are bounded in $L^1$ under a change of the given measure. The technique of changing 
the underlying probability measure is rather popular in stochastic analysis, e.g. in the study 
of semimartingale topologies, see \cite{memin}. It is also widely used in mathematical 
finance where the new probability measure is referred to as the \textit{risk-neutral}
measure, see e.g. \cite{harrison kreps} or \cite{delbaen schachermayer}.

\begin{lemma}
\label{lemma bdd}
Let $\K\subset L^1(\lambda)$ be convex with $\0\in\K$. If $\K$ is bounded in 
$L^0(\lambda)$ then there exists $\mu\in\Prob_*(\lambda)$ such that 
$\K\subset L^1(\mu)$ and $\sup_{k\in\K}\int kd\mu<\infty$.
\end{lemma}

\begin{proof}
Let $\C=\K-\Sim(\A)_+$, pick $A\in\A$ such that $\abs\lambda(A)>0$ 
and fix $x>0$. Suppose that $2x\set A\in\cls\C1\lambda$. For each 
$n\in\N$ there exist then $k_n\in\mathcal K$ and $h_n\in \mathcal C$ such 
that $k_n\ge h_n$ and $\abs\lambda(\abs{h_n-2x\set A})<2^{-n}$. Thus,
\begin{equation}
\label{l}
\begin{split}
\abs\lambda^*(k_n>x)
&\ge
\abs\lambda^*(h_n>x)\\
&\ge
\abs\lambda^*(\abs{h_n-2x}<x)\\
&\ge
\abs\lambda(A)-\abs\lambda^*(A\cap\{\abs{h_n-2x}\ge x\})\\
&\ge
\abs\lambda(A)-\abs\lambda^*(\abs{h_n-2x\set A}\ge x)\\
&\ge
\abs\lambda(A)-x^{-1}\abs\lambda(\abs{h_n-2x\set A})\\
&\ge
\abs\lambda(A)-x^{-1}2^{-n}
\end{split}
\end{equation}
i.e. $\sup_{k\in\K}\abs\lambda^*(\abs k>x)\ge\abs\lambda(A)$. Thus, for $x$ 
sufficiently high, $2x\set A\notin\cls\C1\lambda$. By ordinary separation
theorems there exists a continuous linear functional $\phi$ on $L^1(\lambda)$
such that $\sup_{h\in\C}\phi(h)<c_A<\phi(2x\set A)$ for some 
$c_A>0$. By the inclusion $-\Sim(\A)_+\subset\C$, the functional 
$\phi$ is positive and, by \cite[Theorem 2]{JMAA}, it admits the representation as 
an integral with respect to some $\mu_A\in ba(\lambda)_+$ such that $\mu_A(A)>0$. 
Moreover, since $\phi$ is bounded on any bounded subset of $L^1(\lambda)$
there exists $d_A>0$ such that $\mu\le d_A\abs\lambda$. By normalization we 
may assume $d_A\le1$ and $c_A\le1$. By a finitely additive version of Halmos-Savage 
theorem \cite[Theorem 6]{lebesgue}, the corresponding collection 
$\{\mu_A:A\in\A,\abs\lambda(A)>0\}$ contains a countable subcollection 
$\{\mu_{A_n}:n\in\N\}$ such that, letting $\bar\mu=\sum_n2^{-n}\mu_{A_n}$, then 
$\bar\mu\gg\mu_A$ for all $A\in\A$ with $\abs\lambda(A)>0$ and, therefore, that 
$\bar\mu(A)=0$ if and only if $\abs\lambda(A)=0$. Moreover $\bar\mu\le\abs\lambda$ 
and if $k\in\K$ then $\bar\mu(k)=\sum_n2^{-n}\mu_{A_n}(k)\le1$. It is then enough
to put $\mu=\bar\mu/\norm{\bar\mu}$.
\end{proof}

\begin{remark}
The proof of Lemma \ref{lemma bdd} may be adapted to the case in which $\lambda$
is real valued and additive but not necessarily bounded provided that
each $A\in\A$ with $\abs\lambda(A)=\infty$ admits some $B\in\A$ such that
$B\subset A$ and $\abs\lambda(B)<\infty$. To see this, it is enough to rewrite
the proof upon choosing $A\in\A$ such that $0<\abs\lambda(A)<\infty$. One easily
sees that \eqref{l} still holds as well as the separation argument invoked. We
would obtain a collection $\{\mu_A:A\in\A,0<\abs\lambda(A)<\infty\}$ and, from it,
$\mu=\sum_n2^{-n}\mu_{A_n}$. Then $\mu\ll\lambda$ while $B\in\A$ and $\mu(B)=0$
imply $\lambda(B\cap A)=0$ for all $A\in\A$ with $0<\abs\lambda(A)<\infty$, i.e.
$\abs\lambda(B)=0$.
\end{remark}

Let us remark that Lemma \ref{lemma bdd} requires that $\K$ is convex. This 
assumption is necessary due to the important fact that the convex hull of a bounded 
subset of $L^0(\lambda)$ need not be itself bounded. This is a crucial remark as it 
implies that, generally speaking, the topology of convergence in measure fails to be 
locally convex -- and actually not even linear. This implication makes some useful tools, 
such as separation theorems, simply unavailable.

The next example considers that case of an unbounded set function.

\begin{example}
\label{ex co}
Let $\Omega=\N\times\R_+$, fix $f:[0,1]\to\R_+$ with $\sup_{0\le x\le1}f(x)=\infty$
and $\inf_{0\le x\le1}f(x)=1$ and define $f_n:\Omega\to\R_+$ by letting 
$f_k(n,x)=f(x)$ if $n=k$ or else $0$. Clearly, $\{f_n>c\}=\{n\}\times\{f>c\}$.
Define also 
\begin{equation*}
\A_0
=
\left\{\bigcup_{i=1}^I\{f_{n_i}\ge c_i\}:
n_1,\ldots,n_I\in\N,\ 
c_1,\ldots,c_I\in\R_+,\
I\in\N\right\}
\end{equation*}
and 
\begin{equation*}
m\left(\bigcup_{i=1}^I\{f_{n_i}\ge c_i\}\right)
=
\sum_{i=1}^I(1\vee c_i)^{-1/2}
\quad\text{and}\quad
m(\emp)=0
\end{equation*}
Each pair $A^1,A^2\in\A_0$ admits the representation
$A^j=\bigcup_{i=1}^I\{n_i\}\times\{f\ge c_i(j)\}$ for $j=1,2$,
where $c_1(j),\ldots,c_I(j)\in\R_+\cup\{\infty\}$. It is then easy
to see that
\begin{align*}
m(A^1\cup A^2)+m(A^1\cap A^2)
&=
\sum_{i=1}^I(1\vee(c_i(1)\wedge c_i(2)))^{-1/2}
+\sum_{i=1}^I(1\vee(c_i(1)\vee c_i(2)))^{-1/2}\\
&=
\sum_{i=1}^I(1\vee c_i(1))^{-1/2}+\sum_{i=1}^I(1\vee c_i(2))^{-1/2}\\
&=
m(A^1)+m(A^2)
\end{align*}
Thus, by \cite[Theorems 3.1.6 and 3.2.5]{rao}, $m$ admits an extension 
(still denoted by $m$) as an additive set function on the algebra 
$\A$ generated by $\A_0$. The set $\K=\{f_n:n\in\N\}$ is clearly bounded in 
$L^0(m)$ as  $m(f_n>c)=(1\vee c)^{-1/2}$. However, since the $f_n$'s 
have disjoint support then if $c>1$
\begin{align*}
m\left(\frac 1N\sum_{n=1}^Nf_n\ge c\right)
=
\sum_{n=1}^Nm\left(f_n>cN\right)
=
\frac{1}{\sqrt{c}}\sum_{n=1}^N\frac{1}{\sqrt N}
=
\sqrt {N/c}
\end{align*}
so that $\co(\K)$ is not bounded in $L^0(m)$.
\end{example}

For the case of a bounded additive set function we have a general result that
relates some important topological properties of $L^0(\lambda)$ with the
measure theoretic properties of $\lambda$ (recall the definition of a strongly 
discontinuous set function given in Lemma \ref{lemma cts}).

\begin{theorem}
\label{th co}
The following properties are equivalent: (i) if $\K$ is bounded in $L^0(\lambda)$ then
so is $\co(\K)$, (ii) $\lambda$ is strongly discontinuous, (iii) $L^0(\lambda)$ is a locally 
convex topological vector space.
\end{theorem}

\begin{proof}
(\textit{i})$\Rightarrow$(\textit{ii}). Let $\abs\lambda$ have a strongly continuous part, 
$\abs\lambda_c$. By orthogonality, fix a sequence $\seq E k$ of sets in $\A$ with 
$\abs{\lambda}_c(E_k^c)+\abs{\lambda}_d(E_k)<2^{-k-1}$
and, for each $k$, let $\pi(k)$ be a finite $\A$ partition of $E_k$ such that
$\abs{\pi(k)}>\abs{\pi(k-1)}$ and $\sup_{A\in\pi(k)}\abs{\lambda}_c(A)<2^{-k-1}$
(see \cite[5.2.4]{rao}). Then, $\sup_{A\in\pi(k)}\abs{\lambda}(A)<2^{-k}$.
Define $J(r)=\sum_{k=1}^r\abs{\pi(k)}$ and $k_n=\inf\{k:J(k)>n\}$. Write each 
$\pi(k)$ as $\{A_k^i:i=1,\ldots,I_k\}$ and for each $n\in\N$ define $A(n)=A^i_{k_n}$ 
with $i=n-J(k_n-1)$. It is then clear that $\{A(n):n\in\N\}$ is an enumeration of 
$\{A_k^i:i=1,\ldots,I_k,\ k\in\N\}$. Define $f_n=\abs{\pi(k_n)}^{p+1}\set{A(n)}$, 
with $p>0$. Observe that $\lim_n\abs\lambda(A(n))=0$, as $A(n)\in\pi(k_n)$,
and that $\lim_n\abs{\pi(k_n)}=\infty$. If $n_0$ is large enough so that 
$\sup_{m>n_0}\abs\lambda(A(m))<\varepsilon$ and 
$c>\sup_{i\le n_0}\abs{\pi(k_i)}^{p+1}$, then $\sup_n\abs\lambda(f_n>c)<\varepsilon$.
Thus, $\K=\{f_n:n\in\N\}$ is bounded in $L^0(\lambda)$. However,
\begin{align}
\label{f(pi)}
\abs{\pi(k_n)}^p\set{E_{k_n}}
=
\sum_{A\in\pi(k(n))}\abs{\pi(k_n)}^p\set A
=
\sum_{i=1+J(k_n-1)}^{J(k_n)}\frac{f_i}{J(k_n)-J(k_n-1)}
\in
\co(\K)
\end{align}
so that $\co(\K)$ is not bounded in $L^0(\lambda)$. 

\noindent(\textit{ii})$\Rightarrow$(\textit{i}). Let $\lambda$ be 
strongly discontinuous, i.e. (by Lemma \ref{lemma cts}) let $\abs\lambda$ be of the 
form $\abs\lambda=\sum_{n\ge1}a_n\lambda_n$ with the $\lambda_n$'s being 
$\{0,1\}$-valued and distinct, $a_n>0$ and $\sum_{n\ge1}a_n<\infty$. Observe that
\begin{align*}
\abs\lambda^*(B)
=
\inf_{\{A\in\A:B\subset A\}}\abs\lambda(A)
=
\inf_{\{A\in\A:B\subset A\}}\sum_{n\ge1}a_n\lambda_n(A)
\ge
\sum_{n\ge1}a_n\lambda^*_n(B)
\end{align*}
On the other hand, for each $N\in\N$ there exists a finite partition
$\{F_1,\ldots,F_N\}\subset\A$ such that $\lambda_n(F_n)=1$ for
$n=1,\ldots,N$, \cite[Proposition 5.2.2]{rao}. Thus if we choose $N$ such that 
$\sum_{n>N}a_n<\varepsilon$, and if $B\subset A_n$ and $A_n\in\A$ we have
$B\subset\bigcup_{n=1}^NA_n\cap F_n$ and so
\begin{align}
\label{*}
\sum_{n\le N}a_n\lambda_n(A_n)
=
\sum_{n\le N}a_n\lambda_n\left(\bigcup_{n=1}^NA_n\cap F_n\right)
\ge
\abs\lambda\left(\bigcup_{n=1}^NA_n\cap F_n\right)-\varepsilon
\ge\abs\lambda^*(B)-\varepsilon
\end{align}
Therefore, $\abs\lambda^*=\sum_na_n\lambda_n^*$.
Take $\K$ to be
bounded in $L^0(\lambda)$, and thus in $L^0(\lambda_n)$ for each
$n\in\N$. Given that each $\lambda_n$ is purely atomic and that 
$\lambda_n\ll\abs\lambda$, there exists $c_n>0$ sufficiently high so that
$\sup_{f\in\K}\lambda_n^*(f>c)=0$ whenever $c>c_n$. Take 
$\sum_{i=1}^Ib_if_i\in\co(\K)$ with $f_1,\ldots,f_I\in\K$, $b_1,\ldots,b_I\ge0$
and $\sum_{i=1}^Ib_i=1$. Then, when $c>c_n$ we have
\begin{align*}
\lambda_n^*\left(\sum_{i=1}^Ib_if_i>c\right)
\le
\lambda_n^*\left(\bigcup_{i=1}^I\{f_i>c\}\right)
\le
\sum_{i=1}^I\lambda_n^*\left(f_i>c\right)
=0
\end{align*}
If $N$ is such that $\sum_{n>N}a_n<\varepsilon$ and $c>\sup_{n\le N}c_n$
then from \eqref{*} we conclude
\begin{align*}
\lambda^*\left(\sum_{i=1}^Ib_if_i>c\right)
=
\sum_{n>1}a_n\lambda_n^*\left(\sum_{i=1}^Ib_if_i>c\right)
=
\sum_{n>N}a_n\lambda_n^*\left(\sum_{i=1}^Ib_if_i>c\right)
<
\varepsilon
\end{align*}
so that $\co(\K)$ is bounded in $L^0(\lambda)$.

\noindent(\textit{ii})$\Rightarrow$(\textit{iii}). 
Let $\K_\varepsilon=\{f\in L^0(\lambda):\int\abs f\wedge1d\lambda<\varepsilon\}$.
Under (\textit{ii}) the collection $\{\co(\K_\varepsilon):\varepsilon>0\}$ forms a  
base of absolutely convex, absorbing sets at the origin. Its translates constitute then 
a base for a corresponding locally convex, linear topology. Since 
$\K_\varepsilon\subset\co(\K_\varepsilon)$, this topology is weaker than the topology 
of $\lambda$-convergence. However, given that $\co(\K_\varepsilon)$ is $L^0(\lambda)$ 
bounded, the converse is also true. 

\noindent(\textit{iii})$\Rightarrow$(\textit{i}). Let  $L^0(\lambda)$ be a locally
convex topological vector space and let $\K$ be bounded in $L^0(\lambda)$. There 
will then be a convex open set $C$ around the origin and $\kappa>0$ such that 
$\K\subset\kappa C$ and thus that $\co(\K)\subset\kappa C$ so that $\co(\K)$ is 
bounded in $L^0(\lambda)$.
\end{proof}

For convenience of future reference, let us introduce the class
\begin{equation}
\label{P(K)}
\Prob_*(\lambda;\K)
=
\left\{\mu\in\Prob_*(\lambda):\K\text{ is bounded in }L^1(\mu)\right\}
\end{equation}
The most important consequence of the preceding Lemma \ref{lemma bdd} is the 
following:

\begin{theorem}
\label{th bdd}
Let $\K\subset L^0(\lambda)$ be convex and admit a $\lambda$-a.s. lower bound.
If $\K$ is bounded in $L^0(\lambda)$ then $\Prob_*(\lambda;\K)$ 
is non empty. 
\end{theorem}

\begin{proof}
Let $f\in L^0(\lambda)$ be a $\lambda$-a.s. lower bound for $\K$ and define the sets 
\begin{equation*}
\K_0=\{\alpha(k-f)+\beta\abs f:k\in\K,\ \alpha,\beta\ge0,\ \alpha+\beta\le1\}
\quad\text{and}\quad
\K_1=\{h\wedge k:h\in L^1(\lambda)_+,k\in\K_0\}
\end{equation*}
Observe that
$\K_1$ is a convex subset of $L^1(\lambda)_+$ with $\0\in\K_1$; moreover, 
$\K_1$ is bounded in $L^0(\lambda)$. We deduce from Lemma \ref{lemma bdd} 
the existence of $\mu\in\Prob_*(\lambda)$ such that $\K_1\subset L^1(\mu)$ 
and $\sup_{h\in\K_1}\int hd\mu<\infty$. If $k\in\K$ and $c>0$, then the following
inequality holds $\lambda$-a.s.:
\begin{align*}
\abs k\wedge n
=
\dabs{(k-(f-c))+(f-c)}\wedge n
\le
((k-f)\wedge n)+(\abs f\wedge n)+2c
\end{align*}
Given that $k-f,\abs f\in\K_0$ we conclude that
$\int\abs kd\mu=\lim_n\int(\abs k\wedge n)d\mu\le2[\sup_{h\in\K_1}\int hd\mu+c]$
and the claim follows from the fact that $c$ was chosen arbitrarily.
\end{proof}

Of course there are cases in which the claim of Theorem \ref{th bdd} is rather trivial.
The following are two easy examples.

\begin{example}
\label{ex supremum}
Let $\Omega=\N$ and $\A=2^\N$. Define $\lambda^c,\lambda^\perp\in ba(\A)$
implicitly by letting
\begin{equation*}
\lambda^c(A)=\sum_{k\in A}2^{-k}
\quad\text{and}\quad
\lambda^\perp(A)=\LIM_k\set A(k)
\qquad A\in\A
\end{equation*}
where $\LIM$ denotes the Banach limit. Let $\lambda=\lambda^c+\lambda^\perp$.
Of course, $\lambda^c$ and $\lambda^\perp$ are the countably additive and the
purely finitely additive components of $\lambda$. Define the function
\begin{equation*}
f_n(k)=\exp\left(\frac{n}{1+\abs{n-k}}\right)
\qquad k,n\in\N
\end{equation*}
and let $\K=\co(\{f_n:n\in\N\})$. Observe that 
\begin{align*}
\int f_nd\lambda
\ge
2^{-n}f_n(n)
=
2^{-n}\exp(n)
\end{align*}
so that $\sup_n\int f_nd\lambda=\infty$. Fix $c>1$ and observe that the inequality 
$f_n(k)>\exp(c)$ implies $n,k>c$. Thus if $g=\sum_{i=1}^Ia_if_{n_i}\in\K$ then
\begin{align*}
\{j:g(j)>\exp(c)\}\subset\bigcup_{i=1}^I\{j:f_{n_i}(j)>\exp(c)\}\subset\{j:j>c\}
\end{align*}
so that $\lambda^c(k>c)\le2^{-c}$. On the other hand $\lambda^\perp$ 
does not charge any finite set so that $\sup_{n,\varepsilon}\lambda^\perp%
(f_n>\varepsilon)=0$. The set $\K$ then meets the conditions of Theorem
\ref{th bdd}. Let  $z(k)=2^{-k}$ and observe that $zf_n\le1$ so that
$\mu=\lambda_z$%
\footnote{That is $\mu(A)=\lambda(z\set A)$} 
is such that $\sup_{k\in\K}\int kd\mu\le1$. Moreover,
$z(k)\le1$ so that indeed $\mu\in\Prob_*(\lambda;\K)$. Observe also that
$(\mu+\lambda^\perp)/2$ is another element of $\Prob_*(\lambda;\K)$.
\end{example}

In the preceding example the set $\K$ actually admits a finite supremum. The
following example shows that if the underlying space is countable then the existence 
of a finite supremum is somehow unavoidable under countable additivity, a fact that 
motivates interest for finite additivity.

\begin{example}
\label{ex no supremum}
Let $\Omega$ and $\A$ be as in the previous example and let $\lambda\in ba(\A)_+$
be such that $\lambda^c\ne0$. Observe that for each $A\in\A$,
\begin{equation}
\label{decomp}
\begin{split}
\lambda(A)
&=
\lim_n\lambda(A\cap\{k\le n\})+\lim_n\lambda(A\cap\{k>n\})\\
&=
\sum_{k\in A}\lambda(\{k\})+\lim_n\lambda(A\cap\{k>n\})\\
&=
\lambda^c(A)+\lambda^\perp(A)
\end{split}
\end{equation}
Let $\K$ be the convex hull of a set $\{f_n:n\in\N\}$ of functions $f_n:\N\to\R_+$ and
define $f^*=\sup_nf_n$. For $\K$ to be bounded in $L^0(\lambda)$ it is necessary that 
$\lambda^c(f^*=\infty)=0$. Suppose not. Then there exists $k\in\N$ such that
$\lambda(\{k\})>0$ and for each $j$ there exists $n_j\in\N$ such that $f_{n_j}(k)>2^j$.
But then, if $\mu$ is as in the statement of Theorem \ref{th bdd} one has
$\mu(\{k\})\le2^{-j}\int f_{n_j}d\mu\le2^{-j}\sup_{k\in\K}\int kd\mu$ so that
$\mu(\{k\})=0$ contradicting the inclusion $\mu\in\Prob_*(\lambda)$.
Let now
\begin{equation*}
f_n=\frac{2^{2n}}{1+\abs{k-n}}
\end{equation*}
It is then obvious that $f_n(k)<f_{n+1}(k)$ and that $f^*(k)=\infty$ for 
each $k\in\N$. However, $\K$ is bounded in $L^1(\mu)$ if and only if $\mu$ is 
purely finitely additive. In fact for any such $\mu$ and $N\in\N$ one has 
$\mu(\{1,\ldots,N\}))=0$ so that $\sup_{n,\varepsilon}\mu(f_n>\varepsilon)=0$ 
and thus $\int f_nd\mu=0$. On the other hand, as shown above, if the integrals 
$\int f_nd\mu$ are uniformly bounded this implies $\mu(\{k\})=0$ and, by 
\eqref{decomp}, $\mu^c=0$.
\end{example}

The following result further contributes to understand the role of convexity. 

\begin{corollary}
\label{cor bdd ca}
Let $\K\subset L^0(\lambda)_+$ and assume that $\lambda\in ca(\A)$. Then
$\Prob_*(\lambda;\K)\ne\emp$ if and only if $\co(\K)$ is bounded in $L^0(\lambda)$.
\end{corollary}

\begin{proof}
It is clear that if $\mu$ is as in the claim and $\K$ is a bounded subset of $L^1(\mu)$,
then so is its convex hull $\co(\K)$ which is then bounded in $L^0(\mu)$ too. However,
under the assumption that $\lambda$ is countably additive, $\mu\in\Prob_*(\lambda)$
implies $\lambda\ll\mu$ from which follows that $\co(\K)$ is bounded in $L^0(\lambda)$.
The converse implication follows easily from Theorem \ref{th bdd}.
\end{proof}

\section{Some Topological Implications}
\label{sec topology}

Theorem \ref{th bdd} implies that some subsets of $L^0(\lambda)$ are closed 
in the $L^1(\mu)$ topology with $\mu\in\Prob_*(\lambda)$. 

A first implication of Theorem \ref{th bdd} is the following:

\begin{theorem}
\label{th sep}
Let $\K\subset L^0(\lambda)_+$ be convex and bounded in $L^0(\lambda)$ and define 
\begin{equation}
\label{C}
\C=\left\{f\in L^0(\lambda)_+:f\le g\text{ for some }g\in\K\right\}
\end{equation}
Then, 
\begin{equation}
\cls\C0\lambda\subset\cls\C0\mu=\cls\C1\mu
\qquad\mu\in\Prob_*(\lambda;\C)
\end{equation}
If $\lambda\in ca(\A)$, then
$\cls\C0\lambda=\cls\C1\mu$ for every $\mu\in\Prob_*(\lambda;\K)$.
\end{theorem}

\begin{proof}
By Theorem \ref{th bdd} we can choose $\mu\in\Prob_*(\lambda;\K)$. Then $\C$ 
is a bounded subset of $L^1(\mu)_+$ and thus of $L^0(\mu)$. A relative
comparison of the corresponding topologies shows that $\cls\C0\lambda\subset\cls\C0\mu$
and that $\cls\C1\mu\subset\cls\C0\mu$. It remains to prove that 
$\cls\C0\mu\subset\cls\C1\mu$. Fix $f\in\cls\C0\mu$. Then $f\ge0$ $\mu$-a.s. as 
$\mu^*(\abs{f-h}>c)\ge\mu^*(f<-\varepsilon)$ for $c<\varepsilon$ and $h\in\C$. 
There is a sequence $\seqn f$ in $\C$ that $\mu$-converges to $f$ and thus such that 
$\abs{f_n-f}\wedge k$ converges to $0$ in $L^1(\mu)$ for all $k>0$. The inequality 
$f_n\wedge k-f\wedge k\le\abs{f_n-f}\wedge k$ implies that the sequence 
$\sseqn{f_n\wedge k}$ converges to $f\wedge k$ in $L^1(\mu)$. Thus 
$f\wedge k\in\cls\C1\mu$ and, since $\C$ is bounded in $L^1(\mu)$, $f\in\cls\C1\mu$. 
If $\lambda\in ca(\A)$ and $\mu\in\Prob_*(\lambda)$ then $\lambda\ll\mu$ so that 
$\cls\C0\mu=\cls\C0\lambda$.
\end{proof}

The coincidence of the $L^0(\mu)$ (or even $L^0(\lambda)$) and the $L^1(\mu)$ 
closures may be useful in applications such as the separation of sets, 
a problem which is generally difficult to deal with in $L^0(\mu)$ due to 
the non linear nature of the induced topology. 

\begin{theorem}
\label{th ds}
Let $\A$ be a $\sigma$ algebra and $\lambda\in ca(\A)_+$. Let $\K$ and $\C$ be as 
in Theorem \ref{th sep}, define $\D=\C\cap L^\infty(\lambda)$ and designate by 
$\overline\D^*$ the closure of $\D$ in the weak$^*$ topology of $L^\infty(\lambda)$. 
Then, 
\begin{equation}
\overline\D^*=\cls\D1\mu\cap L^\infty(\lambda)
\qquad
\mu\in\Prob_*(\lambda;\K)
\end{equation}
\end{theorem}

\begin{proof}
Fix $f\in\overline\D^*$. Then $f\ge0$ $\lambda$-a.s. as otherwise
$
\inf_{h\in\D}\int_{\{f<-\varepsilon\}}(h-f)d\lambda
>
\varepsilon\lambda(f<-\varepsilon)
$
which is contradictory. Let $\mu\in\Prob_*(\lambda;\K)$ and denote 
by $Z$ its density with respect to $\lambda$. If $f\in\overline\D^*$ and $\varepsilon>0$,
then there exists $h\in\D$ such that 
$
\varepsilon\ge\int Z(f-h)d\lambda
=
\int fd\mu-\int hd\mu
$, 
so that, $\D$ being bounded in $L^1(\mu)$, the same must be true of $\overline\D^*$. 
By standard arguments, \cite[V.3.13]{bible}, $\cls\D1\mu$ is closed in the weak topology 
of $L^1(\mu)$, 
i.e. the topology induced by $L^\infty(\mu)$. Given the inclusion 
$L^\infty(\mu)\subset L^1(\lambda)$, the restriction of the weak topology of 
$L^1(\mu)$ to $L^\infty(\lambda)$ is weaker than the weak$^*$ topology of 
$L^\infty(\lambda)$, we conclude that $\overline\D^*\subset\cls\D1\mu$. 
Conversely, for each $f\in\cls\D1\mu\cap L^\infty(\lambda)$, there exists a 
sequence $\seqn h$ in $\C$ that converges to $f\in L^\infty(\lambda)$ in the norm 
of $L^1(\mu)$ and is thus $\lambda$-convergent. Upon passing to a subsequence 
and letting $\bar h_n=h_n\wedge\norm f_{L^\infty(\lambda)}$, we conclude that 
$\seqn{{\bar h}}$ converges $\lambda$-a.s. to $f$. Observe that $\bar h_n\in\D$ 
and that, for $g\in L^1(\lambda)$, Lebesgue dominated convergence implies
$\lim_n\int g\bar h_nd\lambda
=
\int gfd\lambda$. We conclude that $f\in\overline\D^*$. 
\end{proof}

Theorem \ref{th ds} thus implies that the weak$^*$ topology on $\D$
is metrizable and may thus, e.g., be described in terms of sequences. Let us also 
mention that the situation described in the statement is crucial in many problems in 
mathematical finance and was first considered by Delbaen and Schachermayer 
\cite[Theorem 2.1]{delbaen schachermayer} who exploited it to establish a special 
version of the no arbitrage principle.

\section{Implications for $L^0(\lambda)$-Valued Operators.}
\label{sec operators}

We establish here some results on $L^1(\mu)$ continuity of $L^0(\lambda)$ valued
operators.

\begin{corollary}
\label{cor TVS}
Let $X$ be a locally convex topological vector space, $V\subset X$ a convex 
neighborhood of the origin and $T:X\to L^0(\lambda)$ a continuous linear operator 
such that $T[V]=\{T(x):x\in V\}$ admits a $\lambda$-a.s. lower bound. There 
exists then $\mu\in\Prob_*(\lambda)$ such that $T[X]\subset L^1(\mu)$ and 
$T:X\to L^1(\mu)$ is continuous.
\end{corollary}

\begin{proof}
Let $\K=T[V]$. Then, $\K$ is convex, bounded in $L^0(\lambda)$ by continuity and 
lower bounded by assumption. By Theorem \ref{th bdd}, there exists
$\mu\in\Prob_*(\lambda;\K)$. Given that 
each neighborhood of the origin is absorbing, this implies that $T[X]\subset L^1(\mu)$. 
Moreover, $T:X\to L^1(\mu)$ is bounded on a neighborhood of the origin and it is thus 
continuous.
\end{proof}

A susbet $U$ of
a vector lattice $X$ is solid if $x\in U$, $y\in X$ and $\abs y\le\abs x$ imply
$y\in U$.

\begin{corollary}
\label{cor lattice}
Let $X$ be a vector lattice with a convex, solid topological basis. A positive, 
continuous operator $T:X\to L^0(\lambda)$ admits $\mu\in\Prob_*(\lambda)$ 
such that $T[X]\subset L^1(\mu)$ and $T:X\to L^1(\mu)$ is continuous.
\end{corollary}

\begin{proof}
Let $V$ be a convex, solid neighborhood of the origin on which $T$ is bounded, 
$V_+=V\cap X_+$ and let $\K=T[V_+]$. Choose, by Theorem \ref{th bdd},
$\mu\in\Prob_*(\lambda;\K)$. If $x\in V$, then $\abs{T(x)}\le T(\abs x)\in\K$,
as $\abs x\in V_+$ whenever $x\in V$. Thus $T[V]$ is a bounded susbet of 
$L^1(\mu)$.
\end{proof}

This last Corollary applies, e.g., to the space $X=\B(S)$ of bounded functions on 
some set $S$ (endowed with the supremum norm).

\begin{corollary}
\label{cor B}
Any positive linear operator $T:\B(S)\to L^0(\lambda)$ admits $\mu\in\Prob_*(\lambda)$ 
such that $T[\B(S)]\subset L^1(\mu)$ and that $T:\B(S)\to L^1(\mu)$ is continuous.
If $\lambda\in ca(\A)$ then $T:\B(S)\to L^0(\lambda)$ is continuous too.
\end{corollary}

\begin{proof}
The unit ball $V$ of $\B(S)$ around the origin is mapped into the set $T[V]\subset[-T(1),T(1)]$ 
which is bounded in $L^0(\lambda)$ and admits $-T(1)$ as a lower bound. By
Theorem \ref{th bdd} there is $\mu\in\Prob_*(\lambda;T[V])$ so that 
$T:\B(S)\to L^1(\mu)$ is continuous. If $\lambda\in ca(\A)$, then $T[V]$, being 
bounded in $L^1(\mu)$, is also bounded in $L^0(\mu)$ and thus in $L^0(\lambda)$ 
as $\mu$ and $\lambda$ are equivalent.
\end{proof}

\begin{example}
\label{ex kernel}
Let $\Sigma$ be an algebra on a given non empty set $S$ and 
$\gamma\in ba(\Sigma)_+$. Consider a map $T:\Omega\times S\to\R_+$ and 
define its $\omega$-section $T_\omega:S\to\R_+$ by letting 
$T_\omega(s)=T(\omega,s)$. Assume that 
(i) $T_\omega\in L^1(\gamma)$ for $\lambda$-almost all $\omega\in\Omega$ 
and (ii) $\int_DT_\omega d\gamma\in L^0(\lambda)$ for each $D\in\Sigma$. 
Then $T$ induces the positive linear operator $\Psi:\B(\Sigma)\to L^0(\lambda)$
defined by letting
\begin{equation}
\Psi(b)=\int bT_\omega d \gamma
\qquad b\in\B(\Sigma)
\end{equation}
By Corollary \ref{cor B}, there exists $\mu\in\Prob_*(\lambda)$ such that
$\Psi:\B(\Sigma)\to L^1(\mu)$ and is continuous in the corresponding topology.
Of course, the map $b\to\int\Psi(b)d\mu$ is then a continuous, positive linear
functional on $\B(\Sigma)$ and admits, by standard results, the representation 
as $\int bd\nu$ with $\nu\in ba(\Sigma)_+$.
\end{example}

Example \ref{ex kernel} easily extends from the random quantities $T_\omega$ 
to the induced vector measure $\int T_\omega d\gamma$.

\begin{theorem}
Let $\Sigma$ be an algebra of subsets of some non empty set $S$ and 
$\Sim(\Sigma,\A)$ the space of $\Sigma$-simple functions with coefficients in 
$\Sim(\A)$ endowed with the norm $\norm f=\sup_{\omega,s}\abs{f(\omega,s)}$. 
Let also $F:\Sigma\to L^0(\lambda)$ be a vector measure. If the expression
\begin{equation}
\label{map}
\int fdF=\sum_{n=1}^Nf_nF(H_n)\qquad 
f=\sum_{n=1}^Nf_n\set{H_n}\in\Sim(\Sigma,\A)
\end{equation}
implicitly defines a continuous linear map of $\Sim(\Sigma,\A)$ into $L^0(\lambda)$
then there exists $\mu\in\Prob_*(\lambda)$ such that the integral $\int fdF$ is
a continuous linear mapping of $\Sim(\Sigma,\A)$ into $L^1(\mu)$.
\end{theorem}

\begin{proof}
$\int fdF:\Sim(\Sigma,\A)\to L^0(\lambda)$ is a continuous linear map if and only if the 
set
\begin{equation*}
I=\left\{\int fdF:f\in\Sim(\Sigma,\A),\ \norm f\le1\right\}
\end{equation*}
is bounded in $L^0(\lambda)$. Observe that $J=\co\{\abs{F(H)}:H\in\Sigma\}\subset I$.
By Theorem \ref{th bdd} there is $\nu\in\Prob_*(\lambda;J)$. We claim that 
$I\subset L^1(\nu)$. In fact, each $f\in\Sim(\Sigma,\A)$ admits the canonical representation 
$\sum_{n=1}^Nf_n\set{H_n}$ where the sets $H_n$ being pairwise disjoint. Thus, 
if $f\in I$ the canonical representation is such that $\sup_n\abs {f_n}\le1$. We
conclude that  
$
N^{-1}\dabs{\int fdF}\
\le
N^{-1}\sum_{n=1}^N\abs{f_n}\abs{F(H_n)}
\le
N^{-1}\sum_{n=1}^N\abs{F(H_n)}
\in J
$.
In addition $I$ is bounded in $L^0(\nu)$ so that, by Lemma \ref{lemma bdd}, 
there is $\mu\in\Prob_*(\nu;I)\subset\Prob_*(\lambda)$, as claimed.
\end{proof}

A classical example of an operator mapping (a subspace of) $\B(S)$ into 
$L^0(\lambda)$ is of course the stochastic integral $\int hdS$ when $S$ is 
a $\lambda$ semimartingale and $\lambda$ a classical probability. The preceding 
Corollaries thus seem to suggest that a meaningful definition of a semimartingale, 
which is beyond the scope of the present paper, may perhaps be obtained 
even when $\lambda$ fails to be countably additive.

\section{$\lambda$-Convergence of Sequences.}
\label{sec convergence}

The same measure change technique exploited above will be applied in this section to
sequences%
\footnote{After this paper was completed I came across the work of Kardaras and 
\v Zitkovi\'c \cite{kardaras zitkovic} that treats some of the topics addressed here 
but only for the countably additive case. 
}. 
In particular we are interested in the possibility of replacing convergence
in measure with $L^1$ convergence. The next Theorem \ref{th memin} establishes 
a finitely additive version of a beautiful result of Memin \cite[Lemma I.4]{memin}, 
widely used in the theory of stochastic integration. Its proof is based on the following 
Lemma, perhaps of its own interest.

\begin{lemma}
\label{lemma memin}
Every sequence $\seqn f$ in $L^0(\lambda)$ that $\lambda$-converges 
to $0$ admits a subsequence $\seqnk f$ such that the following set is bounded 
in $L^0(\lambda)$:
\begin{equation}
\label{K}
\K=\left\{\sum_{k=1}^K\alpha_k2^k\abs{f_{n_k}}:
\alpha_1,\ldots,\alpha_K\ge0,\
\sum_{k=1}^K\alpha_k\le1,\
K\in\N
\right\}
\end{equation}

\end{lemma}

\begin{proof}
Choose iteratively $n_k>n_{k-1}$ such that
$
\sup_p\abs\lambda^*\left(\abs{f_{n_k+p}}>2^{-k}\right)\le2^{-k}
$
and put $g_k=2^k\abs{f_{n_k}}$. Fix $c>1$ and let $\seq\alpha k$ be a sequence
of positive numbers with finitely many non null terms and $\sum_k\alpha_k\le1$. 
Exploiting the subadditivity of the set function $\abs\lambda^*$ we obtain the 
following inequality:
\begin{align*}
\abs\lambda^*\left(\sum_k\alpha_kg_k>2c\right)
&\le
\abs\lambda^*\left(\sum_{k<k_0}\alpha_kg_k>c\right)
+\abs\lambda^*\left(\sum_{k\ge k_0}\alpha_kg_k>c\right)\\
&\le
\abs\lambda^*\left(\sum_{k<k_0}\alpha_k\abs{f_{n_k}}>2^{-k_0}c\right)
+\sum_{k\ge k_0}\abs\lambda^*\left(\abs{f_{n_k}}>2^{-k}\right)\\
&\le
\abs\lambda^*\left(\sup_{k<k_0}\abs{f_{n_k}}>2^{-k_0}c\right)
+2^{-k_0+1}
\end{align*}
If $k_0$ and $c$ are large enough so that $2^{-k_0+1}<\varepsilon/2$ and
$\abs\lambda^*(\sup_{k<k_0}\abs{f_{n_k}}>2^{-k_0}c)%
<\varepsilon/2$, then, $\abs\lambda^*\left(\sum_k\alpha_kg_k>2c\right)<%
\varepsilon$. 
\end{proof}

We say that a sequence $\seqn f$ is $\lambda$-Cauchy if $f_n\in L^0(\lambda)$ 
for every $n\in\N$ and 
\begin{equation}
\label{cauchy}
\lim_n\sup_{p,q}\abs\lambda^*\left(\dabs{f_{n+p}-f_{n+q}}>c\right)=0
\qquad c>0
\end{equation}

\begin{theorem}
\label{th memin}
For each $i\in\N$, let $\seqn {f^i}$ be $\lambda$-Cauchy and 
$\seqn {h^i}$ $\lambda$-convergent to $0$. Let $\K_1$ be a 
convex, bounded subset of $L^0(\lambda)_+$. There exists 
$\mu\in\Prob_*(\lambda;\K_1)$ and a sequence $\seq nk$ of positive
integers increasing to $\infty$ such that
\begin{equation}
\lim_k\sup_{p,I\in\N}
\int\sum_{i=1}^I\sum_{j=k}^{k+p}
\left[\dabs{f^i_{n_j}-f^i_{n_{j+1}}}+\dabs{h^i_{n_j}}\right]d\mu
=0
\end{equation}
\end{theorem}

\begin{proof}
By a diagonal argument, it is possible to fix $\seq nk$ so that 
\begin{equation}
\label{memin u}
\sup_{p,q}\abs
\lambda^*\left(
\sum_{i\le k}\left[\dabs{f^i_{n_{k+p}}-f^i_{n_k+q}}+\dabs{h^i_{n_{k+p}}}\right]
>2^{-k}
\right)
\le
2^{-k}
\end{equation}
Let 
$
\hat g_k=\sum_{i\le k}[\abs{f^i_{n_k}-f^i_{n_{k+1}}}+\abs{h^i_{n_k}}]
$. 
The sequence $\seq{\hat g}k$ is $\lambda$-convergent to $0$ so that, by Lemma 
\ref{lemma memin} and by letting $g_k=2^k\hat g_k$, the set $\K_2$ of finite sums 
of the form $\sum_{k=1}^K\alpha_kg_k$ with $\alpha_1,\ldots,\alpha_K\ge0$ and 
$\sum_{k=1}^K\alpha_k\le1$ is bounded in $L^0(\lambda)_+$. Given that $\K_1$
and $\K_2$ are bounded and convex, then so is $\K=\co(\K_1\cup\K_2)$, as remarked 
in the introduction. But then, Theorem \ref{th bdd} implies the existence of 
$\mu\in\Prob_*(\lambda,\K)$. Then, from 
\begin{align*}
\sum_{j=k}^{k+p}\hat g_j
=
2^{-(k-1)}\sum_{j=k}^{k+p}2^{-(j-k+1)}g_j
\quad\text{and}\quad
\sum_{j=k}^{k+p}2^{-(j-k+1)}g_j\in\K
\end{align*}
we conclude that
$
\lim_k\sup_p\int\sum_{j=k}^{k+p}\hat g_jd\mu
\le
\lim_k2^{-k}\sup_{h\in\K}\mu(h)
=0
$.
The proof is complete upon noting that  
$
\sum_{j=k}^{k+p}[\abs{f^i_{n_j}-f^i_{n_{j+1}}}+\abs{h^i_{n_j}}]
\le
\sum_{j=k}^{k+p}\hat g_j
$,
for $i=1,\ldots,k$.
\end{proof}

One should note that the sequence $\seq nk$ in the claim does not depend on 
$i\in\N$. Observe also that each sequence $\seqnk{f^i}$ is Cauchy in $L^1(\mu)$ 
and each $\seqnk{h^i}$ is convergent in $L^1(\mu)$. Due to incompleteness 
of $L^p$ spaces under finite additivity, the existence of a sequence which is 
Cauchy in $L^1(\mu)$ may appear an unsatisfactory conclusion. Incompleteness 
is amended, however, if we replace each $f_n\in L^1(\mu)$ with its isomorphic 
image in $ba(\lambda)$, as the sequence $\sseqn{\mu_{f_n}}$ converges in 
norm to some $m\in ba(\mu)\subset ba(\lambda)$ although $m$ may not
be representable as a $\mu$ integral.

In the classical theory of stochastic processes this result has a number of
applications. If, e.g., $(M_t:t\in\R_+)$ is a non negative supermartingale
on some filtration $(\A_t:t\in\R_+)$ of sub $\sigma$ algebras of $\A$,
then, by Doob's convergence Theorem, $M$ converges to a $\lambda$-a.s. 
finite limit $M_\infty$. By Theorem \ref{th memin}, we can replace 
$\lambda$ with an equivalent probability measure $\mu$ such that $M$ 
converges to $M_\infty$ in $L^1(\mu)$ and is therefore uniformly integrable 
with respect to $\mu$. 

This conclusion is based on the strict interplay between convergence in measure
and pointwise convergence which is a distinguishing feature of countable
additivity. Under finite additivity, however, the situation may be more complex.
The following example, making use of the notation employed in the proof of
Theorem \ref{th co}, illustrates some possible pathologies.

\begin{example}
Assume that $\lambda$ is not strongly discontinuous and borrow from the proof 
of Theorem \ref{th co} the definition of $\seq E k$, $\sseq{\pi(k)}k$ and 
$\sseqn {A(n)}$. Let $\seqn g$ be a sequence that $\lambda$-converges to $g$
(and therefore bounded in $L^0(\lambda)$) and let 
$f_n=\abs{\pi(k_n)}\set{A(n)}g_{\abs{\pi(k_n)}}$.
Then, $\seqn f$ is $\lambda$-convergent to $0$ while
\begin{equation}
h_k
=
\sum_{i=J(k-1)+1}^{J(k)}\frac{f_i}{J(k)-J(k-1)}
=
\sum_{i=J(k-1)+1}^{J(k)}\set{A(i)}g_{\abs{\pi(k_i)}}
=
g_{\abs{\pi(k)}}\set{E_k}
\end{equation}
Then $h_k\in\co\{f_k,f_{k+1},\ldots\}$;
moreover, $\seq h k$ is $\lambda$-Cauchy but does not $\lambda$-converge to $0$. 
In fact
if $\lambda$ is strongly continuous -- and so $E_k=\Omega$ -- then $\seq h k$
$\lambda$-converges to $g$. If $\lambda$ has s strongly discontinuous part then
it may well not converge at all. Take the case in which 
$E_k\subset E_{k+1}\uparrow\Omega$ and $g_n=g$. Then $h_k$ converges
pointwise to $g$ but 
$
\abs\lambda^*(\abs{g-h_k}>c)
\ge
\abs\lambda(E_k^c)
\ge
\lambda_d(\Omega)-2^{-k}
$.
\end{example}

In the countably additive setting, Kardaras and \v Zitkovi\' c 
\cite[Example 1.2]{kardaras zitkovic}
construct the example of a sequence converging in measure from which it is
possible to extract via convex combinations further sequences which converge
in measure to any, preassigned measurable function. 

Theorem \ref{th memin} allows to replace measure convergence with $L^1$ 
convergence. We can also obtain conditions under which a $\lambda$-convergent
sequence also converges $\lambda$-a.s..

We start proving the following preliminary result.

\begin{lemma}
\label{lemma liminf}
Let $f,f_n\in L^0(\lambda)$ for $n=1,2,\ldots$ be such that
\begin{equation}
\label{liminf}
\lim_k\abs\lambda^*\left(\inf_{n>k}f_n<f-c\right)=0
\qquad c>0
\end{equation}
Then $\liminf_nf_n\ge f$, $\lambda$-a.s..
\end{lemma}

\begin{proof}
Assume \eqref{liminf}, fix $c>0$ and let $g_k=\sum_{n\le k}2^n(f-c-f_n)^+$ and 
$g=\sum_n2^n(f-c-f_n)^+$. Then, $\{\abs{g-g_k}>c\}\subset\bigcup_{n>k}\{f_n<f-c\}%
=\left\{\inf_{n>k}f_n<f-c\right\}$. By assumption, $\seq g k$ $\lambda$-converges to 
$g\in L^0(\lambda)$ so that $\abs\lambda^*(g=\infty)=0$. Moreover, since
$g_k$ converges to $g$ monotonically too then
$f_n\ge f-c-2^{-n}g$
so that $\{\liminf_nf_n<f-c\}\subset\{g=\infty\}$.
\end{proof}

Lemma \ref{lemma liminf} provides a sufficient criterion for the existence of a
measurable lower bound to a sequence. It also provides a sufficient condition
for $\lambda$-a.s. convergence:

\begin{theorem}
\label{th lim}
Let $\seqn f$ be a sequence in $L^0(\lambda)$ and define 
$g_k=\inf_{n,m>k}(f_n-f_m)$. If $\seq gk$ $\lambda$-converges to $0$, 
then $\liminf_nf_n=\limsup_nf_n$, $\lambda$-a.s..
\end{theorem}

\begin{proof}
Fix $c>0$. By assumption, $\lim_j\abs\lambda^*(g_j<-c)=0$. Lemma \ref{lemma liminf} 
thus implies that $\liminf_nf_n-\limsup_mf_m=\liminf_jg_j\ge0$, $\lambda$-a.s..
\end{proof}

It is important to remark that, contrary to the classical case, the random quantity 
$g_k$ in the claim is not generally measurable and so neither is the $\lambda$-a.s.
limit of the sequence $\seqn f$. The need to consider convergence properties of
non measurable elements arises also in other parts of probability, see \cite{berti rigo}
for an illustration and references. 

It is also easily seen that in the classical case any $\lambda$-convergent sequence 
admits a subsequence that meets the criterion of Theorem \ref{th lim} which may
thus may be regarded as a partial, finitely additive version of the classical property by 
which each $\lambda$-converging sequence admits a subsequence converging 
$\lambda$-a.s..

To conclude, in the following Theorem \ref{th komlos} we prove a finitely additive 
version of a subsequence principle that is often useful in applications. It is related 
to a well known result of Koml\'{o}s \cite{komlos}. It proves that it is possible, 
given any $\lambda$-bounded sequence, to build a sequence which is $\lambda$%
-Cauchy -- although not necessarily $\lambda$-convergent.

If $\seqn f$ is a sequence, denote by $\Gamma(f_1,f_2,\ldots)$ the family of 
all those sequences $\seqn h$ such that $h_n\in\co\{f_n,f_{n+1},\ldots\}$ for 
all $n\in\N$.
\begin{theorem}
\label{th komlos}
Let $\seqn f$ be a sequence in a convex subset $\K$ of $L^0(\lambda)_+$. 
(i) If $\K$ is bounded in $L^1(\lambda)$ then $\Gamma(f_1,f_2,\ldots)$ 
contains a $\lambda$-Cauchy sequence; (ii) if $\K$
is bounded in $L^0(\lambda)$ then $\Gamma(f_1,f_2,\ldots)$ contains
a sequence which is Cauchy in $L^1(\mu)$ for some $\mu\in\Prob_*(\lambda;\K)$.
\end{theorem}

\begin{proof}
With no loss of generality assume $\lambda\ge0$ and let $\K$ be bounded in 
$L^1(\lambda)$. Consider the sequence $\seqn\lambda$ with 
$\lambda_n=\lambda_{f_n}$. By \cite[Theorem 5]{lebesgue} there exists 
$\mu_n\in\co\{\lambda_n,\lambda_{n+1},\ldots\}$ such that the sequence 
$\sseqn{\mu_n\wedge a\lambda}$ is norm convergent for all $a\in\R_+$. Let 
$h_n\in\co\{f_{n},f_{n+1},\ldots\}$ be such that $\mu_n=\lambda_{h_n}$. 
Clearly, $\lambda_{h_n\wedge a}\le\lambda_{h_n}\wedge a\lambda$. In fact 
if $\sseq{h_{n,r}}r$ is a sequence in $\Sim(\A)$ converging to $h_n$ in 
$L^1(\lambda)$, then, using norm convergence,
\begin{align*}
\lambda_{h_n}\wedge a\lambda
=
\lim_r(\lambda_{h_{n,r}}\wedge a\lambda)
=
\lim_r\lambda_{h_{n,r}\wedge a}
=
\lambda_{h_n\wedge a}
\end{align*}
the last line following from the inequality 
$\abs{x_1\wedge a-x_2\wedge a}\le\abs{x_1-x_2}$. Thus the sequence
$\sseqn{h_n\wedge a}$ is Cauchy in $L^1(\lambda)$ for all $a\in\R_+$ 
so that
\begin{align*}
\lambda^*(\abs{h_n\wedge a-h_m\wedge a}>c)&
\ge
\lambda^*(\abs{h_n-h_m}>c;h_m\vee h_n\le a)\\
&\ge
\lambda^*(\abs{h_n-h_m}>c)-\abs\lambda^*(h_m\ge a)-\lambda^*(h_n\ge a)
\end{align*}
and thus $
\lambda^*(\abs{h_n-h_m}>c)
\le
2a^{-1}\sup_{k\in\K}\int kd\lambda+c^{-1}\int\dabs{h_n\wedge a-h_m\wedge a}
$.
We can then choose the sequence $\seq nk$ such that $n_k\ge k$ and that
\begin{equation*}
\sup_{p,q}\lambda^*\left(\abs{h_{n_{k+p}}\wedge a-h_{n_{k+q}}\wedge a}>2^{-k}\right)
\le
2^{-k}
\end{equation*}
The subsequence $\seqnk h$ is thus $\lambda$-Cauchy. If, $\K$ is just bounded in 
$L^0(\lambda)$, then (ii) follows from Theorem \ref{th memin} upon passing to a 
further subsequence, still denoted by $\seqnk h$ for convenience. The proof is 
complete if we let $g_k=h_{n_k}$ upon noting that indeed 
$\seq gk\in\Gamma(f_1,f_2,\ldots)$.
\end{proof}

Claim (\textit{ii}) of Theorem \ref{th komlos} becomes considerably stronger under 
countable additivity, when completeness of $L^p$ spaces may be invoked. The 
sequence $\seqn g$ would then converge in $L^1(\mu)$ and, upon passing to 
a subsequence if necessary, a.s. too. The statement asserting that,  by taking 
convex combinations, it is possible to extract from a sequence of positive, measurable 
functions another sequence that converges a.s., is often referred to as Koml\'os 
lemma (see  \cite[Theorem 1]{komlos}) and has become widely used in the literature. 
The interplay between convergence in measure and a.s. convergence is crucial to 
this end and requires countable additivity. When $\lambda$ is just finitely additive, 
Theorem \ref{th komlos} may be useful to obtain from a sequence converging a.s. 
a further sequence that converges a.s. and is Cauchy in measure. 

As a final application of Theorem \ref{th komlos} we obtain the following:
\begin{corollary}
Let $\varphi:L^1(\lambda)\to\R$ be uniformly continuous and 
$\K$ a convex, uniformly integrable subset of $L^1(\lambda)_+$. For each 
sequence $\seqn f$ in $\K$ there exists a sequence $\seqn h$ in 
$\Gamma(f_1,f_2,\ldots)$ such that $\varphi(bh_n)$ converges for every 
$b\in\B(\A)$.
\end{corollary}

\begin{proof}
Let $\seqn h$ be the $\lambda$-Cauchy sequence of Theorem \ref{th komlos}. By 
uniform integrability, $\lim_{a\to\infty}\sup_{f\in\K}
\norm{f-(f\wedge a)}_{L^1(\lambda)}=0$; 
by continuity, the limit $\lim_{a\to\infty}\varphi(f\wedge a)$ exists uniformly in $f\in\K$. 
Thus, for each $b\in\B(\A)$ we obtain
\begin{align*}
\limsup_{n,m}\dabs{\varphi(bh_n)-\varphi(bh_m)}
&=
\limsup_ {n,m}\lim_{a\to\infty}\dabs{\varphi(b(h_n\wedge a))-\varphi(b(h_m\wedge a))}\\
&=
\lim_{a\to\infty}\limsup_{n,m}\dabs{\varphi(b(h_n\wedge a))-\varphi(b(h_m\wedge a))}\\
&=0
\end{align*}
where we exploited \cite[I.7.6]{bible}, the inequality
$\dabs{(h_{n+p}\wedge a)-(h_{n+q}\wedge a)}\le\dabs{h_{n+p}-h_{n+q}}\wedge a$ 
and the fact that as $\abs{h_{n+p}-h_{n+q}}\wedge a$ tends to $0$ in $L^1(\mu)$ as 
$n$ approaches $\infty$.
\end{proof}

\acknowledgement{I am deeply grateful to an anonymous referee for a number of 
helpful suggestions and for pointing out several mistakes in a previous draft.}

\end{document}